\documentclass[reqno]{amsart}
\usepackage{amssymb,url}
 \usepackage[abbrev]{amsrefs}

\DeclareMathOperator{\Eu}{Eu}
\DeclareMathOperator{\Vol}{Vol}
\DeclareMathOperator{\M}{M}
\DeclareMathOperator{\MP}{MP}
\DeclareMathOperator{\codim}{codim}
\DeclareMathOperator{\Grass}{Grass}
\DeclareMathOperator{\Conv}{Conv}

\theoremstyle{plain}% default
 \newtheorem{thm}{Theorem}%[section]
 \newtheorem{cor}[thm]{Corollary}

 \theoremstyle{definition}

\def\mylistparam
 {\setbox0=\hbox{\rm(1)\kern0.5em}
  \setlength{\labelwidth}{\wd0}
  \setlength{\topsep}{\smallskipamount}%
  \setlength{\itemsep}{0pt}\setlength{\parsep}{0pt}%
  \setlength{\itemindent}{0pt}%
  \setlength{\leftmargin}{\parindent}%
  \addtolength{\leftmargin}{\wd0}%
 }%
 {\begin{list}{\rm(\arabic{enumi})}%
  {\usecounter{enumi}\mylistparam}%
  }%
 {\end{list}}

\begin{document}

\author[R. Piene]{Ragni Piene}
\address{Department of Mathematics\\
University of Oslo\\ 
PO Box 1053, Blindern\\
NO-0316 Oslo, Norway}
\email{ragnip@math.uio.no}

\title{Chern--Mather classes of toric varieties}

\begin{abstract}
The purpose of this short note is to prove 
a formula
for the Chern--Mather classes of a toric variety in terms of its orbits and the local Euler obstructions at general points of each orbit (Theorem 2).  We use the general definition of the Chern--Schwartz--MacPherson classes (see \cite{MR0361141}) and their special expression in case of a toric variety (see \cite{MR1197235}). As a corollary, we obtain  a formula by Matsui--Takeuchi \cite[Corollary 1.6]{MR2737807}. Alternatively, one could deduce the formula of Theorem \ref{CM} from the Matsui--Takeuchi formula, by using our general result \cite[Th\'eor\`eme 3]{MR1074588} for the degree of the polar varieties in terms of the Chern--Mather classes.
\end{abstract}

\maketitle

We first recall the definition of the Chern--Mather class $c^{\M}(X)$ of an $n$-dimensional variety $X$. Let $\widetilde X\subseteq  \Grass_n (\Omega_X^1)$ denote the Nash transform of $X$, i.e., $\widetilde X$ is the closure of the graph of the rational section of $\Grass_n (\Omega_X^1)$ given by the locally free rank $n$ sheaf $\Omega_X^1|_{X_{\rm sm}}$. 
We set $c^{\M}(X):=\nu_*(c(\Omega^\vee)\cap [X])$, where $\Omega$ is the tautological sheaf on $\Grass_n (\Omega_X^1)$ and 
$\nu\colon \widetilde X\to X$. 

 The \emph{polar loci} of an $n$-dimensional projective variety $X\subset \mathbb P^N$ are defined as follows: Let $L_k\subset \mathbb P^N$ be a linear subspace of codimension $n-k+2$. The polar locus of $X$ with respect to $L_k$ is
 \[M_k:=\overline{\{x\in X_{\rm sm}\,|\, \dim (T_{X,x}\cap L_k)\ge k-1\}},\]
 where $T_{X,x}$ denotes the projective tangent space to $X$ at the (smooth) point $x$. (For other interpretations of $M_k$, see e.g. \cite{MR1074588}.) The rational equivalence classes $[M_k]$ are independent of $L_k$, for general $L_k$, and (in 1978) we showed the following:
 
 \begin{thm}\label{CP}\cite[Th\'eor\` eme 3]{MR1074588}
 The polar classes of $X$ are given by
  \begin{equation}\label{CP1}
 [M_k]=\sum_{i=0}^k (-1)^{i}\binom{n-i+1}{n-k+1}h^{k-i}\cap c^M_{i}(X),\end{equation}
 and, reciprocally, the Chern--Mather classes of $X$ are given by
 \begin{equation}\label{CP2}
 c^M_k(X)=\sum_{i=0}^k (-1)^{i}\binom{n-i+1}{n-k+1}h^{k-i}\cap [M_{i}],\end{equation}
 where $h$ is the class of a hyperplane.
 \end{thm}

Recall (see \cite{MR0361141}) that the Chern--Schwartz--MacPherson  class of  $X$ is defined by 
\[c^{\MP}(X):=c^{\M}\circ T^{-1}(\mathbf 1_X).\]
Here we define $c^{\M} \colon Z(X)\to A(X)$ by $c^{\M}(\sum n_i V_i)=\sum n_i c^{\M}(V_i)$, and
 $T$ is the isomorphism from the group of cycles $Z(X)$ to the group of constructible functions on $X$, given by 
\[T(V)(x)= \Eu_{V}(x),\]
where $\Eu_V$ denotes the the constructible function whose value at a point $x\in X$ is equal to the local Euler obstruction of $V$ at $x$ (hence is 0 if $x\notin V$). Note that $\Eu_V(x)=1$ if $x\in V$ is a smooth point, but that the converse is false (this was first observed in  \cite[Example, pp. 28--29]{MR1074588}).
The Chern--Schwartz--MacPherson classes are invariant under homeomorphisms, whereas the Chern--Mather classes are invariant under generic linear projections (since the polar varieties are \cite[Thm. 4.1, p.~269]{MR510551}.

\medskip

In what follows we shall consider \emph{toric} varieties, defined as follows.
Let $A\subset \mathbb Z^n$ be a set of $N+1$ points such that the polytope $P:=\Conv (A)\subseteq \mathbb R^n$ has dimension $n$.
Let $X:=X_A\subset \mathbb P^N$ denote the corresponding toric variety. Let $\{X_\alpha\}_\alpha$ denote the orbits of the torus action on $X$. The classes $[\overline X_\alpha]$ generate the Chow ring $A(X)$ \cite[5.1, Prop., p.~96]{MR1234037}. Moreover,  $\overline X_\alpha$ is the toric variety corresponding to the lattice points $A_\alpha:=A\cap F_\alpha$, where $F_\alpha$ is the face of $P$ corresponding to the orbit $X_\alpha$. It was shown in \cite[Th\'eor\`eme]{MR1197235} that the Chern--Schwartz--MacPherson class of $X$ is given by
\begin{equation}\label{1}
c^{\MP}(X)=\sum_\alpha [\overline X_\alpha].
\end{equation}

The purpose of this note is to prove Theorem \ref{CM} below, using (\ref{1}).  As a corollary we obtain formulas for the ranks (degrees of the polar varieties) of $X$, in particular  the formula of  \cite[Corollary 1.6]{MR2737807}, hence we have an alternative proof of this result. Observe that if we instead \emph{assume}  \cite[Corollary 1.6]{MR2737807}, then we can deduce Theorem \ref{CM} by using \cite[Th\'eor\`eme 3]{MR1074588}.

\begin{thm}\label{CM}
The Chern--Mather class of the toric variety $X$ is equal to \[c^{\M}(X)=\sum_\alpha \Eu_X(X_\alpha)[\overline X_\alpha],\] where the sum is taken over all orbits $X_\alpha$ of the torus action on $X$, and where  $\Eu_X(X_\alpha)$ denotes the value of the local Euler obstruction of $X$ at a point in the orbit $X_\alpha$.
\end{thm}

\begin{proof} 

Write $T^{-1}(\mathbf 1_X)=X+\sum_\alpha a_\alpha \overline X_\alpha$ for some $a_\alpha \in \mathbb Z$, so that $\mathbf 1_X=T(X)+\sum_\alpha a_\alpha T(\overline X_\alpha)$, where the sums are over $\alpha$ such that $\overline X_\alpha \neq X$. Then for $x\in X_\beta$, we get
\begin{equation}\label{2}
1=\Eu_X(X_\beta)+\sum_{\alpha \succ \beta} a_\alpha \Eu_{\overline X_\alpha}(X_\beta),
\end{equation}
where the sum is over all orbits $X_\alpha$ such that $X\neq \overline X_\alpha\supset X_\beta$. We also have
 \[c^{\MP}(X)=c^{\M}\circ T^{-1}(\mathbf 1_X)=c^{\M}(X)+\sum_\alpha a_\alpha c^{\M}(\overline X_\alpha).\]
 Using (\ref{1}), this gives
 \begin{equation}\label{3}c^M(X)=[X]+\sum_\alpha [\overline X_\alpha]-\sum_\alpha a_\alpha c^{\M}(\overline X_\alpha),\end{equation}
 where again the sum is over all $\alpha$ such that $\overline X_\alpha\neq X$.

We shall use induction on the dimension of $X$. If $\dim X=1$, then there are two $0$-dimensional orbits $x_1$ and $x_2$. Thus (\ref{2}) gives $1=\Eu_X(x_i)+a_i$, for $i=1,2$, so that $a_i=1-\Eu_X(x_i)$. Hence (\ref{3}) gives 
\[c^M_1(X)=\sum [x_i]-\sum (1-\Eu_X(x_i))[x_i]=\sum \Eu_X(x_i)[x_i],\]
which is what we wanted to show.

Assume now that the theorem holds for toric varieties of dimension $< \dim X$. Then for each $\overline X_\alpha\neq X$ we can write 
\[c^M(\overline X_\alpha)=\sum_{\beta\prec \alpha} \Eu_{\overline X_\alpha}(X_\beta) [\overline X_\beta], \]
where the sum is over all $\beta$ such that $X_\beta\subset \overline X_\alpha$. From (\ref{3}) we get
\[c^M(X)=[X]+\sum_\alpha [\overline X_\alpha]-\sum_\alpha a_\alpha\sum_{\beta\prec \alpha} \Eu_{\overline X_\alpha}(X_\beta)[\overline X_\beta].\]
Rewriting the last double sum as $\sum_\beta (\sum_\alpha a_\alpha \Eu_{\overline X_\alpha}(X_\beta))[\overline X_\beta]$ and applying (\ref{2}) gives the formula of the theorem.
\end{proof}

Let $\mu_k:=\deg M_k$ denote the degrees of the polar varieties of $X$. Applying the  equality (\ref{CP1}) of Theorem \ref{CP} we obtain:

\begin{thm}
The degrees of the polar varieties of the toric variety $X$ are given by
\[\mu_k=\sum_{i=0}^k (-1)^i \binom{n-i+1}{n-k+1} \sum_\alpha \Eu_X(X_\alpha) \Vol_\mathbb Z(F_\alpha),\]
where the second sum is over all $\alpha$ such that $X_\alpha$ has codimension $i$ in $X$, and $ \Vol_\mathbb Z(F_\alpha)$ denotes the lattice volume of the face $F_\alpha$ of $P$ corresponding to the orbit $X_\alpha$.
\end{thm}

\begin{cor}
[Matsui--Takeuchi {\cite[Corollary 1.6]{MR2737807}}]
Assume  the dual variety of $X\subset \mathbb P^N$ is a hypersurface. Then its degree is given by
\[\deg X^\vee=\sum_{F_\alpha\preceq P} (-1)^{\codim F_\alpha}(\dim F_\alpha+1) \Eu_X(X_\alpha) \Vol_\mathbb Z(F_\alpha),\]
where $X_\alpha$ denotes the orbit in $X$ corresponding to the face $F_\alpha$ of $P$.
\end{cor}

\begin{proof}
In this case the degree of the dual variety is equal to $\mu_n$.
\end{proof}

\medskip

\textbf{Examples.} 
The toric varieties we consider need not be normal, in particular the set of lattice points $A$ need not be equal to the set of lattice points in $P=\Conv (A)$. Note that we can view $X_A$ as a (toric) linear projection of $X_P$. When this projection is ``generic'', the Chern--Mather classes of $X_A$ are just the pushdowns of the Chern--Mather classes of $X_P$ \cite[Corollaire, p. 20]{MR1074588}. Here are two simple examples.
\medskip

1)  Let $A=\{(0,0), (0,1), (1,1), (2,0)\}$. Then $X_A\subset \mathbb P^3$ is a cubic surface with a double line with two pinch points, and it is the projection of  a rational normal surface of type $(1,2)$. The closure of the orbit of $X_A$ corresponding to the line segment $[(0,0), (2,0)]$ has normalized lattice volume $1$ and local Euler obstruction 2. The three other 1-dimensional orbits are smooth and have lattice volume 1. Moreover, as shown in \cite[p. 29]{MR1074588}, the local Euler obstruction at a pinch point is 1. Hence we get
\[ \deg X_A^\vee = 3\cdot 3-2(2\cdot 1+3)+4=3.\]
Since any toric hypersurface $X_A$, where $A$ is not a pyramid, is selfdual \cite{MR2835342}, this is of course no surprise. Note that we also get $\deg X_P^\vee =3$.
\medskip

2) Let $A=\{(0,0), (1,1), (0,2), (3,0)\}$. Then $X_A\subset \mathbb P^3$ is a (non-generic) toric linear projection of the weighted projective space $X_P=\mathbb P(1,2,3)\subset \mathbb P^6$. In this case $\deg X_A^\vee =6$, whereas $\deg X_P^\vee=7$. Note that for $X_A$, the local Euler obstructions at the 1-dimensional orbits are 1, 2, and 3, whereas all the three 0-dimensional orbits have local Euler obstruction 0.  (For more examples of explicit computations of the local Euler obstruction for toric varieties, especially in the case of weighted projective spaces, see \cite{N}.)

\medskip

\textbf {Remark.}
There has recently been a renewed interest in Chern--Mather classes and polar varieties, in particular related to the concept of Euclidean distance degree. This includes other types of polar varieties  (see the survey \cite{MR3335572} and the references given there). For a cycle theoretic approach, see \cite{Aluffi};  for applications, see \cite{MR3451425} for the general case and \cite{HS} for the toric case.

\end{document}